\documentclass[11pt, reqno]{amsart}

\title[Peripheral Poisson boundary]{Peripheral Poisson boundary}

\allowdisplaybreaks[1]

\usepackage{relsize}
\usepackage{amsmath,amsfonts,amssymb,amsthm,mathrsfs}
\usepackage{hyperref}
\usepackage{cleveref}
\usepackage[margin=2.5cm]{geometry}
\usepackage{setspace}
\usepackage{mathtools}

\emergencystretch=\maxdimen
\usepackage{setspace}
\usepackage{xcolor}


\allowdisplaybreaks[1]

\numberwithin{equation}{section}
\newtheorem{theorem}{\bf Theorem}[section]
\newtheorem{lemma}[theorem]{\bf Lemma}

\newtheorem{defin}[theorem]{\bf Definition}
\newtheorem{example}[theorem]{\bf Example}
\newtheorem{corollary}[theorem]{\bf Corollary}

\newcommand{\ot}{\otimes}

\newcommand{\N}{\mathbb{N}}
\newcommand{\M}{\mathbb{M}}

\newcommand{\T}{\mathbb{T}}

\newcommand{\ol}{\overline}

\newcommand{\norm}[1]{\| #1 \|}

\makeatletter \@namedef{subjclassname@2020}{\textup{2020}
Mathematics Subject Classification} \makeatother

\begin{document}

\title[Peripheral Poisson Boundary]{Peripheral Poisson Boundary}

\author{B. V. Rajarama Bhat}
\address{Indian Statistical Institute, Stat Math Unit, R V College Post, Bengaluru, 560059, India.}
\email{bvrajaramabhat@gmail.com, bhat@isibang.ac.in}

\author{Samir Kar}
\address{Indian Statistical Institute, Stat Math Unit, R V College Post, Bengaluru, 560059, India.}
\email{msamirkar@gmail.com}

\author{Bharat Talwar}
\address{ Department of Mathematics, Nazarbayev University, Nur-Sultan, Kazakhstan.}
\email{btalwar.math@gmail.com, bharat.talwar@nu.edu.kz}



\maketitle

\begin{abstract}
 It is shown that the operator space generated
by peripheral eigenvectors of a unital completely positive map on a
von Neumann algebra has a $C^*$-algebra structure. This extends the
notion of non-commutative Poisson boundary by including the point
spectrum of the map contained in the unit circle. The main
ingredient is dilation theory. This theory provides a simple formula
for the new product. The notion has implications to our
understanding of quantum dynamics. For instance, it is shown that the peripheral Poisson boundary remains invariant in discrete quantum dynamics.

\end{abstract} \maketitle

\section{Introduction}

Unital completely positive (UCP) maps on von Neumann algebras are
considered as non-commutative (or quantum) Markov maps. Their fixed
points are called harmonic elements. In general, this collection of
harmonic elements is not an algebra. However, it is possible to
introduce a new product called Choi-Effros product (\cite{CE}) which
makes it a von Neumann algebra.

Inspired from the theory of classical random walks on groups Izumi
\cite{Izumi2002} called this algebra (or its concrete realization)
as noncommutative Poisson boundary. This concept has  attracted a
lot of interest in recent years ( See \cite{Izumi2004},
\cite{Izumi2012}, \cite{BBDR}, \cite{SS}, \cite{Vaes}, \cite{DP},
\cite{Ka}). The space of harmonic elements is just the eigenspaces
with respect to the eigenvalue one. In this short note, we look at
the operator space formed by eigenvectors corresponding to
eigenvalues on the unit circle for quantum Markov maps, i.e., for
UCP maps. Like in the case of harmonic elements,  to begin with we
have only an operator space, which may not be an algebra. But we can
impose a new product to get a $C^*$-algebra. Eigenvalues on the unit
circle are known as peripheral eigenvalues and the corresponding
eigenvectors are called peripheral eigenvectors. Therefore we are
providing an extension of Choi-Effros product to the operator space
spanned by peripheral eigenvectors. To spell out our results in more
detail we need some notation.

 All the Hilbert spaces considered will be complex and separable
whose inner product $\langle \cdot , \cdot \rangle $ is anti-linear
in the first variable.  We denote the algebra of all bounded
operators on a Hilbert space ${\mathcal H}$ by $B({\mathcal H})$. In
formulae, we indicate strong operator topology limit by `$s-\lim $'.
The unit circle, $\{z\in {\mathbb C}: |z|=1\}$ is denoted by
${\mathbb T}$.

Let $\mathcal H$ be a  Hilbert space and let ${\mathcal A}\subseteq
B({\mathcal H})$ be a von Neumann algebra. Suppose $\tau :{\mathcal
A}\to {\mathcal A}$ is a normal unital completely positive map. It
is well-known that the fixed point space of $\tau $ (also known as
the space of harmonic elements of $\tau $),
$$F (\tau ):=E_1(\tau )=\{ x\in {\mathcal A}: \tau (x)=x\}$$
admits a  product `$\circ $', called Choi-Effros product on $F(\tau
)$, which will make it a von Neumann algebra. Note that there is no
modification in the norm or the involution. Concrete realization of
this algebra is known as the non-commutative Poisson boundary of
$\tau .$  For any $\lambda \in {\mathbb C}$ take
$$E_{\lambda }(\tau ) :=\{ x\in {\mathcal A}: \tau (x) =\lambda
x\}.$$ The set of all  $\lambda \in {\mathbb T}$ for which
$E_{\lambda }(\tau )$ is non-trivial is known as the peripheral
point spectrum of $\tau $ and the corresponding eigenvectors are
known as peripheral eigenvectors. Take $$E(\tau ) = \mbox{span}\{x:
x\in E_{\lambda }(\tau ), ~\mbox{for some}~\lambda \in {\mathbb
T}\}.$$

In our main theorem (Theorem \ref{main}) it is shown that the  norm
closure of $E(\tau )$ is a $C^*$-algebra with respect to a new
product. Here too, the norm and the involution are not modified.
Only the product is replaced with a new product. The norm is of
course uniquely determined as a $*$-algebra can admit at most one
norm which makes it a $C^*$-algebra. We may call this $C^*$-algebra
as the peripheral Poisson boundary.

In fact, we have a very simple formula for the product (See Theorem
\ref{formula}).  No ultra-filters are required to compute this
product. It is also shown that if the UCP map leaves a faithful
state invariant then the new product is same as the original
product. It is not hard to see that the action of the UCP maps
yields an automorphism of the peripheral Poisson boundary (Theorem
\ref{automorphism}).

In Section 3, we present a few examples to illustrate the theory. We
see that the Poisson boundary product of Toeplitz operators, which
comes from $L^{\infty}({\mathbb T})$, has a very natural extension
to the peripheral Poisson boundary by incorporating a family of
unitaries parametrized by the unit circle.  It is well known that
the norm of a Toeplitz operator is same as the $L^{\infty}$ norm of
its symbol. In equation (\ref{norm}) we have an interesting
extension of this result.  One of the surprising features we see in
this example is that the algebra we get is not a von Neumann
algebra. The equation (\ref{norm}) does not extend to the strong
closure.

The importance of peripheral spectrum in quantum dynamics is well
known and has its origin in classical Perron-Frobenius theory ( See
\cite{Al}, \cite{Ar1}, \cite{EvansKrohn}, \cite{FEU} \cite{FOR},
\cite{GrohMZ}).  The main point  is that generally the peripheral
spectrum encodes the persistent/recurrent part of the dynamics and
the rest has only some transient component.
Below in \Cref{automorphism}, we answer a question raised in \cite{FV} by showing that restriction of the UCP map to the peripheral part is an automorphism in our setting.

For a UCP map $\tau $, if we look at the Poisson boundaries of $\tau
, \tau ^2, ...$, they keep changing, for instance if $e^{\frac{2\pi
i}{k}}$ (for some $k>1$), is in the point spectrum of $\tau $, the
corresponding eigenspace gets included among harmonic elements of
$\tau ^k$, but they are not part of the set of harmonic elements of
$\tau ^{l}$ when $l$ is not a multiple of $k.$ In other words the
point spectrum keeps `rotating' on the unit circle when we take
powers of $\tau $. However, as we see in Theorem \ref{stability},
the peripheral Poisson boundary of $\tau $ is same as that of $\tau
^k$ for all $k\geq 1.$ Most of the papers dealing with peripheral
point spectrum of UCP maps have additional conditions such as finite
dimensionality of the algebra and/or existence of an invariant
state. We do not impose any such restrictions.

The main technical tool we need is the following minimal dilation
theorem for unital completely positive maps. In the following for
Hilbert spaces ${\mathcal H}\subset {\mathcal K}$, any operator
$x\in B({\mathcal H})$ is identified with the operator $pxp$ in
$B({\mathcal K})$, where $p$ is the orthogonal projection of
${\mathcal K}$ onto ${\mathcal H}.$

\begin{theorem}\label{Dilation} (\cite{Bha96}, \cite{Bha99}, and
\cite{BS})  Let ${\mathcal A}\subseteq B({\mathcal H})$ be a von
Neumann algebra and let $\tau :{\mathcal A}\to {\mathcal A}$ be a
normal unital completely positive map. Then there exists a triple
$({\mathcal K}, {\mathcal B}, \theta )$, where  (i)  ${\mathcal K}$
is a Hilbert space containing ${\mathcal H}$ as a closed  subspace;
(ii) ${\mathcal B}\subseteq B({\mathcal K})$ is a von Neumann
algebra, satisfying ${\mathcal A}=p{\mathcal B}p$ where $p$ is the
orthogonal projection of ${\mathcal K}$ onto ${\mathcal H}$; (iii)
$\theta :{\mathcal B} \to {\mathcal B}$ is a normal, unital
$*$-endomorphism; (iv) (dilation property): $$\tau ^n(x)=p\theta
^n(x)p,~~\forall x\in {\mathcal A}, n\in {\mathbb Z}_+;$$ (v)
(minimality of the space) ${\mathcal
K}=\overline{\mbox{span}}\{\theta ^n(x_n)\theta
^{n-1}(x_{n-1})\ldots \theta (x_1)h: x_1, \ldots , x_n\in {\mathcal
A}, h\in {\mathcal H}, n\in {\mathbb Z}_+\};$ (vi) (minimality of
the algebra) ${\mathcal B}$ is the von Neumann algebra generated by
$\{\theta ^n(x): n\in {\mathbb Z}_+, x\in {\mathcal A}\}.$ (vii)
(uniqueness) If $({\mathcal K}_1, {\mathcal B}_1, {\theta }_1)$ is
another such triple then there exists a unitary $U:{\mathcal K}\to
{\mathcal K}_1$ satisfying (a) $Uh=h$ for all $h\in {\mathcal H}$;
(b) ${\mathcal B}_1=\{ UbU^*: b\in {\mathcal B}\}$; (c) $\theta
_1(x)= U\theta (x)U^*, \forall x\in {\mathcal A}.$
\end{theorem}

 This theorem was first proved for normal CP maps on $B({\mathcal
H})$ in \cite{Bha96} and then for CP maps on general $C^*$-algebras
in \cite{Bha99}. Here we have taken the setting of von Neumann
algebras with normal UCP maps. This was first done in \cite{BS},
using a Hilbert $C^*$-module approach. It is to be noted that, these
articles  handle discrete (where the semigroup under consideration
is ${\mathbb Z}_+$) and the continuous case (where the semigroup is
${\mathbb R}_+$) at one go and the main focus was on the continuous
case.  Presently we need only the much simpler discrete case. A very
compact and neat exposition of this case can be seen in the Appendix
of \cite{DP}. It makes use of  repeated application of Stinespring's
theorem and is essentially  equivalent to first constructing the
weak Markov flow and then writing down the endomorphism. The
increasing family of projections constructed here  (which increases
to identity)  is the filtration of quantum Markov process theory
(See \cite{BP94}, \cite{BP95}). Some alternative
versions/presentations can be seen in (\cite{Ar}, \cite{MS02},
\cite{Sa}). The literature on dilation theory of UCP maps is vast
and we have cited only a small sample relevant for our current
purposes. Here we have confined ourselves to the von Neumann algebra
version, as that is convenient and  seems to be of most interest.

 We will call the triple $({\mathcal K}, {\mathcal B}, \theta )$ of
this theorem as the minimal dilation of $\tau .$ We will repeatedly
use the following two facts of dilation theory:

(i) For  every $n\in {\mathbb Z}_+$, $p\leq \theta ^n(p),$ and
$\theta ^n(p)$ increases in strong operator topology to the identity
of ${\mathcal K}.$ (ii) For any $z\in {\mathcal B}$, $p\theta (z)p=
\tau (pzp).$

We have restricted our focus to dynamics in discrete time in this
article. But it is clear that it is possible to develop analogous
theory in continuous time.

\section{Main Theorem}

We assume the following setup and notation for whole of this
Section. Let  ${\mathcal A}\subseteq B({\mathcal H})$ be a von
Neumann algebra and let $\tau :{\mathcal A}\to {\mathcal A}$ be a
normal unital completely positive map. Let $({\mathcal K}, {\mathcal
B}, \theta )$, be the minimal dilation of $\tau $ guaranteed by
Theorem \ref{Dilation}. Note that the orthogonal  projection $p$ of
${\mathcal K}$ onto  ${\mathcal H}$  is the identity of ${\mathcal
A}$.

Our first objective is to connect the peripheral eigenvectors of
$\tau $ with that of $\theta .$ Then we use the $*$-endomorphism
property of $\theta $ to get a $C^*$-algebra structure on the
operator space generated by these eigenvectors. To begin with we
have the following simple observation that peripheral eigenvectors
of $\tau $ can be lifted to that of $\theta $ and peripheral
eigenvectors of $\theta $ can be compressed to that of $\tau .$

\begin{lemma}\label{TIsOnto}
For every $x \in {\mathcal A}$ satisfying $\tau(x) = \lambda x$ with
$\lambda \in \T$, there exists  unique $\hat{x} \in {\mathcal B}$
such that $\theta(\hat{x}) = \lambda \hat{x}$ and $p \hat{x} p = x$.
It is given by \begin{equation}\label{lifting} \hat{x}= s-\lim
_{n\to \infty}\lambda^{-n} \theta^n(x).\end{equation} Conversely if
$\hat{x}\in {\mathcal B}$ satisfies $\theta (\hat{x}) = \lambda
\hat{x}$ for some $\lambda \in {\mathbb C}$, then $\tau (x)=\lambda
x$, where $x=p\hat{x}p.$
\end{lemma}
\begin{proof}
Let us prove the existence of $\hat{x}$ as above.  Note that here,
as $x\in {\mathcal A}=p{\mathcal B}p$, we have $x=pxp$.   Take
$x_0=x$. For $n\geq 1,$ set $x_n =\lambda ^{-n}\theta ^n (x)=
\lambda^{-n} \theta^n (pxp)$. We claim that $\{x_n\}_{n\geq 0}$
forms a {\em martingale\/}  with respect to the filtration $\{\theta
^m(p):m\geq 0\}$, that is,  for $0 \leq m \leq n,$
\begin{equation}\label{consistency}\theta^m(p) x_n \theta^m(p) =
x_m.\end{equation} This can be seen easily by dilation property:
\begin{eqnarray*}
\theta^m(p) x_n \theta^m(p) &=&  \theta^m(p) \lambda^{-n} \theta^n
(pxp) \theta^m(p)
= \lambda^{-n} \theta^m(p)  \theta^n (pxp) \theta^m(p)\\
&=& \lambda^{-n}  \theta^m (p \theta^{n-m}(pxp) p) = \lambda^{-n}  \theta^m (p \tau^{n-m}(x)p) \\
&=& \lambda^{-n}  \theta^m (p \lambda^{n-m} xp) = \lambda^{-n} \lambda^{n-m}  \theta^m (pxp) \\
&=& \lambda^{-m} \theta^m (pxp)  = x_m.
\end{eqnarray*}
As $|\lambda |=1$ and $\theta $ is contractive,
$\{x_n\}_{n \geq 0}$ is a bounded sequence. Now as $\{ \theta
^m(p)\}_{m\geq 0} $ is an increasing family of projections
increasing to identity, for $h\in {\mathcal K}$ and $\epsilon >0$,
we can choose $m$ such that $\|h-\theta ^m(p)h\|<\epsilon$. Taking
$h_0=\theta ^m(p)h$, from (\ref{consistency}), for $n\geq k \geq m$,
$x_kh_0$ and $x_nh_0- x_kh_0$ are mutually orthogonal and hence
$\{\|x_nh_0\|^2 \}_{n\geq m}$ is an increasing sequence bounded by
$\|x\|^2\|h_0\|^2$. Further, equation (\ref{consistency})implies,
$\|(x_n-x_m)h_0\|^2= \|x_nh_0\|^2-\|x_mh_0\|^2,~~\forall n\geq m.$
It follows that $\{x_nh_0\}_{n\geq 0}$ is a Cauchy sequence. As
$\epsilon >0$ was arbitrary,  $\{x_nh\}_{n\geq 0}$ is also a Cauchy
sequence. Taking $\hat{x}h= \lim _{n\to \infty}x_nh$, as
$\|x_nh\|\leq \|x\|\|h\|,$ for all $n$, we get a bounded operator
$\hat{x}=s-\lim _{n\to \infty}x_n$ in ${\mathcal B}$ satisfying
$\theta ^m(p)\hat{x}\theta ^m(p)=x_m.$ Note that
\begin{eqnarray*}
\theta(\hat{x})& =& \theta(s-\lim _{n\to \infty} x_n) =
\theta(s-\lim _{n\to \infty} \lambda^{-n} \theta^n (pxp))  =  s-\lim
_{n\to \infty} \lambda^{-n} \theta^{n+1} (pxp)\\& =& \lambda (s-\lim
\lambda^{-n-1} \theta^{n+1} (pxp)) = \lambda \hat{x}.\end{eqnarray*}
Further, for $m=0$, taking limit as $n\to \infty$ in equation
(\ref{consistency}), $p\hat{x}p=x.$ This proves the existence of
$\hat{x}$

 Let us now prove the uniqueness. Suppose that there exists $y' \in
{\mathcal B}$ such that $\theta(y') = \lambda y'$ and $py'p = x$.
Then since $\theta ^n(p)$ increases to identity as $n$ increases to
infinity, we get
\begin{eqnarray*} y'& =& s-\lim _{n\to \infty} \theta^n(p) y'
\theta^n(p) = s-\lim _{n\to \infty} \lambda^{-n} \theta^n(p)
\theta^n(y') \theta^n(p)\\& = & s-\lim _{n\to \infty} \lambda^{-n}
\theta^n(py'p) = s-\lim _{n\to \infty}\lambda^{-n} \theta^n(x) =
\hat{x}.\end{eqnarray*}  This completes the proof of the first part.
Now suppose $\hat{x}\in {\mathcal B}$ satisfies $\theta (\hat{x})
=\lambda \hat{x}$ for some $\lambda \in {\mathbb C}$.  Take
$x=p\hat{x}p.$ As $\theta (p)\geq p$, we get
$$\tau(x)= \tau(p\hat{x}p)= p\theta (p\hat{x}p)p= p\theta (p)\theta(\hat{x})\theta(p)p=
p\theta(\hat{x})p=\lambda p\hat{x}p=\lambda x.$$ \end{proof}
It is to be noted that in this Lemma, $\hat{x}\neq 0$ as
$p\hat{x}p=x$. So we have  lifted eigenvectors of $\tau $ to that of
$\theta $. In the computations we have crucially used the fact that
we are considering peripheral eigenvalues. Later on in the Hardy
space example of Section 3, we see  that it is possible to have
$\lambda $ in the point spectrum of $\tau $, which are not in that
of $\theta $ for $|\lambda |<1$.

Consider the compression map $T:{\mathcal B}\to {\mathcal A}$
defined by
$$T(z)=pzp.$$
Clearly $T$ is completely positive and unital. The previous Lemma
shows that for $\lambda \in {\mathbb T}, $ $T$ from  $E_{\lambda
}(\theta )$ to $E_{\lambda }(\tau )$ is a bijection. Now we wish to
show that $T$ maps $E(\theta )$ isometrically to $E(\tau ).$  For
the purpose, we need the following result on rotations of the
$d$-torus. It is perhaps a folklore result in the theory of
dynamical systems. However, we could not find a suitable reference.
For readers convenience we present it in a purely Hilbert space
language and provide an elementary operator theoretic proof. It
maybe noted that the standard result that irrational rotations on
the torus have dense orbit (See Lemma 8.1.17 of \cite{GKPT})  is a
simple special case of this result.

\begin{lemma}\label{dynamics}
Suppose $U$ is a unitary on a finite dimensional Hilbert space. Then
there exists a subsequence of $\{U^n\}_{n\in {\mathbb N}}$
converging to the identity in norm. In particular, for $d\in
{\mathbb N}$, if $\lambda _1, \ldots , \lambda _d$ are some
$d$-points in the unit circle ${\mathbb T}$, then there exists a
subsequence of $\{(\lambda _1^n, \ldots , \lambda _d^n)\}_{n\in
{\mathbb N}}$ converging to $(1, \ldots , 1)$ in ${\mathbb C}^d.$
\end{lemma}
\begin{proof}
Since the set of unitaries on a finite dimensional Hilbert space is
closed and bounded and hence compact, the sequence $\{U^n\}_{n\in
{\mathbb N}}$ has a convergent subsequence, say $\{U^{n_m}\}_{m\in
{\mathbb N}}.$ Then by Cauchy property for $\epsilon >0$, there
exists $K$, such that for $l>m>K$, $\|U^{n_l}-U^{n_m}\|<\epsilon .$
This implies,
$$\|U^{n_l-n_m}-I\|\leq \|(U^{n_l}-U^{n_m})U^{-n_m}\|\leq
\|U^{n_l}-U ^{n_m}\|<\epsilon .$$ Now the result is clear. The
second part follows from considering the  diagonal unitary with
diagonal entries equal to $\lambda _1, \ldots , \lambda _d.$
\end{proof}

We remark in the passing that Arveson (\cite{Ar1}, Lemma 2.9) has a
result in the converse direction, namely that in any Banach space if
identity is in the strong closure of  $\{U, U^2, \ldots \}$ for some
linear contraction $U$, then $U$ is an automorphism. We will not
 need this result.

Consider the operator spaces:
$${\mathcal P}(\tau ):=\overline{ E(\tau )}= \overline{{\mbox{span}}}\{x\in {\mathcal A}:
\tau (x)=\lambda x, \lambda \in {\mathbb T}\};$$
$${\mathcal P}(\theta ):= \overline{E(\theta )}= \overline{{\mbox{span}}}\{x\in {\mathcal B}:
\theta (x)=\lambda x, \lambda \in {\mathbb T}\},$$ where the
closures are taken with respect to the operator norm. Clearly
${\mathcal P}(\tau ), {\mathcal P}(\theta )$ are operator spaces. In
fact as they are unital and $*$-closed they form operator systems.
Since $\theta $ is an endomorphism, if $\theta (x)=\lambda x$ and
$\theta (y)=\mu y$, then $\theta (xy)= \lambda .\mu (xy).$ This
shows that ${\mathcal P}(\theta )$ is a unital $C^*$-algebra. Here
is our main result.

\begin{theorem}\label{main}
 Let ${\mathcal A}\subseteq B({\mathcal
H})$ be a von Neumann algebra and let $\tau :{\mathcal A}\to
{\mathcal A}$ be a normal unital completely positive map. Let
$({\mathcal K}, {\mathcal B}, \theta )$, be the minimal dilation of
$\tau .$ Let $T$ denote the compression map $z\mapsto pzp$
restricted to ${\mathcal P}(\theta ).$  Then the completely positive
map $T$ maps the $C^*$-algebra ${\mathcal P}(\theta )$ isometrically
and bijectively to ${\mathcal P}(\tau ).$  In particular, setting
\begin{equation}\label{newproduct}
x\circ y := T (T^{-1}(x)T^{-1}(y)) ,~~\forall x,y \in {\mathcal
P}(\tau ),\end{equation} makes $({\mathcal P}(\tau ), \circ )$  a
unital $C^*$-algebra. Moreover, the map $T$ is a complete isometry.
\end{theorem}

\begin{proof}
From Lemma \ref{TIsOnto} we know that $T$ maps $E_{\lambda }(\theta
)$ bijectively to $E_{\lambda }(\tau )$ for every $\lambda \in
{\mathbb T}.$ Consequently, $T$ maps $E(\theta)$ surjectively to
$E(\tau)$. Recall that for any linear map eigenvectors
corresponding to distinct eigenvalues are linearly independent.
Hence any non-zero element $y$ in $E(\theta)$ can be written
uniquely as $y = \sum _{j=1}^dy_j$ for some $0\neq y_j\in E_{\lambda
_j}(\theta )$ for some distinct $\lambda _1, \ldots , \lambda _d$ in
${\mathbb T}$, $d\in {\mathbb N}.$  Take $x_j=py_jp, ~1\leq j\leq d,
$ and $x=pyp$. As $p$ is a projection,  $\|x\|\leq \|y\|.$ The
construction in the proof of Lemma \ref{TIsOnto},  tells us that
$$y_j= s-\lim _{n\to \infty}\lambda _j^{-n}\theta ^n(x_j), ~~1\leq j\leq d.$$
Hence, \begin{equation}\label{limit} y = s-\lim_{n\to
\infty}\sum_{j=1}^d \lambda _j^{-n}\theta ^n(x_j).\end{equation}
 Consider any $h$ in the Hilbert space $ {\mathcal K}$ with $\|h\|= 1$ and let $\epsilon >0$.
By the convergence in  \ref{limit}, there exists $N$ such that for
all $n\geq N$,
$$\|yh-\sum _{j=1}^d\lambda _j^{-n}\theta ^n(x_j)h\|<\frac{\epsilon}{2} .$$
In view of Lemma \ref{dynamics}, we can find $m>N$ such that
$$|\ol{\lambda _j}^{m}-1|< \frac{\epsilon}{2\sum _{k=1}^d\|x_k\|}, ~~\forall 1\leq j\leq d.$$
As $\theta $ is an endomorphism and $\|h\|=1$, $\|\theta
^m(x_j)h\|\leq \|x_j\|$ and hence
$$\|\sum _{j=1}^d(\lambda _j^{-m}-1)\theta
^m(x_j)h\|<\frac{\epsilon}{2}.$$ Combining with the previous
inequality at $n=m$,
$$\|yh\|\leq \| \sum _{j=1}^d\theta ^m(x_j)\|+\epsilon =\| \theta ^m(\sum
_{j=1}^dx_j)\| +\epsilon =\|\theta ^m(x)\|+\epsilon \leq
\|x\|+\epsilon .$$ As this is true for all $\epsilon >0$, $\|y\|\leq
\|x\|\leq \|y\|.$ This proves that $T$ is isometric on $E(\theta ).$
In particular, $T$ is injective and also it extends to an isometric
bijection from ${\mathcal P}(\theta )$  to ${\mathcal P}(\tau ).$
The second part is clear, as ${\mathcal P}(\theta )$ is a  unital
$C^*$-algebra. The complete isometry property of $T$ is shown in the
Appendix. \end{proof}

\begin{defin}
 Let ${\mathcal A}\subseteq B({\mathcal
H})$ be a von Neumann algebra and let $\tau :{\mathcal A}\to
{\mathcal A}$ be a normal unital completely positive map. Then the
$C^*$-algebra ${\mathcal P}(\tau )$ constructed above is called the
{\em peripheral Poisson boundary\/} of $\tau .$
\end{defin}

We could have called this as `{\em non-commutative peripheral
Poisson boundary}'. The word `non-commutative' has been dropped for
brevity as it is clear from the context. We have defined the
peripheral Poisson boundary using the minimal dilation. It is
possible to write down the algebra structure without referring to
the minimal dilation as below. However we do not know as to how to
prove the existence of strong operator topology limit in this
theorem without using dilation. Initial formulae for Choi-Effros
product included limits over ultra-filters. It was W. Arveson and M.
Izumi who first saw that a much simpler formula can be provided
using dilation theory (See the Appendix of  \cite{Izumi2012}). Here
we have similar formula for peripheral Poisson boundary.

\begin{theorem}\label{formula}
Under the setting of Theorem \ref{main}, for $x,y\in {\mathcal A}$
with $\tau (x)=\lambda x, \tau(y)=\mu y,~\lambda , \mu \in {\mathbb
T}$, \begin{equation}\label{productformula} x\circ y = s-\lim _{n\to
\infty} (\lambda \mu )^{-n}\tau ^n(xy).\end{equation} If there
exists a faithful state left invariant by $\tau $, then $x\circ
y=xy.$
\end{theorem}

\begin{proof}
Take $\hat{x}=T^{-1}(x)$ and $\hat{y}=T^{-1}(y)$ in ${\mathcal B}.$
Then from Lemma \ref{TIsOnto}, we know that $\hat{x}=s-\lim _{n\to
\infty} \lambda ^{-n}\theta ^n(x)$ and $\hat{y}=s-\lim _{n\to \infty
}\mu ^{-n}\theta ^n(y)$. Since strong operator topology convergence
respects products of sequences, $\hat{x}\hat{y}=s-\lim _{n\to
\infty}{(\lambda \mu )}^{-n}\theta ^n(xy).$ Therefore,
\begin{eqnarray*}
x\circ y&=&  p\hat{x}.\hat{y}p.
\end{eqnarray*}
To see the second part let $\phi $ be a faithful state left
invariant by $\tau $. Consider $x\in E_{\lambda }(\tau ), y\in
E_{\mu }(\tau )$ with $\lambda , \mu \in {\mathbb T}.$ As $\tau $ is
UCP, for any $y\in {\mathcal A}$, $\tau (y^*y)\geq \tau (y)^*\tau
(y).$ Hence
$$\phi (\tau (y^*y)-\tau (y)^*\tau (y))=\phi (\tau (y^*y)-|\mu
|^2y^*y)= \phi (y^*y)-\phi (y^*y)=0.$$ The faithfulness of $\phi $
implies $\tau (y^*y)-\tau(y)^*\tau(y)=0.$ Then by dilation property
$p\theta (y)^*[p+(1-p)]\theta (y)p-p\theta (y)^*p\theta (y)p=0$ or
$(1-p)\theta (y)p=0$. Therefore, $\tau (xy)=p\theta
(x)[p+(1-p)]\theta (y)p=\tau(x)\tau (y)=\lambda \mu xy$ and the
result follows from the first part. \end{proof}

This result about UCP maps preserving faithful states has
applications to the theory of quantum channels. For instance, it
shows that in such cases products of two peripheral eigenvectors is
again a peripheral eigenvector.  We may also note the reverse
implication that if the product in the peripheral Poisson boundary
does not match with the ordinary product then no faithful state is
left invariant by the UCP map.

\begin{corollary}
If the von Neumann algebra ${\mathcal A}$ is abelian then the
peripheral Poisson boundary of $\tau :{\mathcal A}\to {\mathcal A}$
is also abelian.

\end{corollary}

\begin{proof}
If the algebra is abelian we get $x\circ y =y\circ x$ from equation
(\ref{productformula}). \end{proof}

\begin{corollary}
If $x\in E_{\lambda }(\tau )$ and $y\in E_{\mu}(\tau )$ then $x\circ
y\in E_{\lambda .\mu }(\tau )$. In particular, if $\lambda .\mu$ is
not in the point spectrum of $\tau$, then $x\circ y=0.$
\end{corollary}

\begin{proof}
This follows as we have lifts $\hat{x}\in E_{\lambda }(\theta ),
\hat{y}\in E_{\mu}(\theta )$ and consequently $\hat {x}.\hat{y} \in
E_{\lambda .\mu }(\theta ).$ \end{proof}

We remark that instead of considering the whole group ${\mathbb T}$,
we may consider any subgroup $G$ of ${\mathbb T}$ and define
$$E_G(\tau )= \mbox{span}\{x\in E_{\lambda }(\tau ):
\lambda \in G\}.$$ Like in the main theorem, $T$ maps $E_{G}(\theta
)$ isometrically and bijectively to $E_{G}(\tau ).$ Following the
same procedure ${\mathcal P}_G(\tau ):=\overline{E_G(\tau )}$
becomes a unital $C^*$-algebra, which is in fact the
$C^*$-subalgebra of ${\mathcal P}(\tau )$ generated by $E_G(\tau ).$
The following case is of special interest.

\begin{defin}\label{k-cyclic}
For $k\in {\mathbb N}$,  $k$-cyclic Poisson boundary of UCP map
$\tau :{\mathcal A}\to {\mathcal A}$ is defined as the $C^*$-algebra
${\mathcal P}_{G_k}(\tau )$, where $G_k$ is the group $\{\omega ^j:
0\leq j\leq {k-1}\}$ with $\omega = e^{\frac{2\pi i}{k}}.$
\end{defin}

By considering the trivial group $\{1\}$, we get the usual
noncommutative Poisson boundary as a unital $C^*$-subalgebra of
${\mathcal P}(\tau ).$ Here are some additional properties of
peripheral eigenvectors as elements of  the peripheral Poisson boundary.

\begin{theorem}
Let $E_{\lambda }(\tau )$ for $\lambda \in {\mathbb T}$ be the
spaces as above.

(i) Define $I_{\lambda }(\tau )= \overline{\mbox{span}}\{x^*\circ y:
x, y\in E_{\lambda }(\tau )\}.$ Then $( I_{\lambda }(\tau ), \circ
)$ is a closed two sided ideal of the Poisson boundary
$C^*$-algebra, $(F(\tau), \circ )$.

(ii) $E_{\lambda }(\tau  )$ with $\langle x, y\rangle :=x^*\circ y$
(and natural left and right actions) is a two-sided Hilbert
$C^*$-module over the $C^*$-algebra $(F(\tau ), \circ )$, with inner
products taking values in the ideal $(I_{\lambda}(\tau ), \circ).$

(iii) The mapping $x\mapsto x^*$ from $E_{\lambda}(\tau )$ to
$E_{\bar{\lambda } }(\tau )$ is an anti-linear isomorphism.
\end{theorem}

\begin{proof}
All the computations can be done using the minimal dilation $\theta
$ as $E_{\lambda }(\theta )$ and $E_{\lambda }(\tau )$ are
isomorphic and the product operation on ${\mathcal P}(\tau )$ is
borrowed from ${\mathcal P}(\theta ).$ Now the results (i)-(iii)
follow from the $*$-homomorphism property of $\theta .$ \end{proof}

\begin{corollary}\label{isometries} For $\lambda \in {\mathbb T}$ if there exist  $v_1, v_2\in E_{\lambda }(\tau )$,
such that $v_1^*\circ v_2=1$,  then
$$E_{\lambda}(\tau )= \{ x\circ v_2: x\in F(\tau )\} =\{x\circ v_1: x\in F(\tau )\}. $$
Similarly, if there exist $v_1, v_2\in E_\lambda(\tau )$ such that
$v_1\circ v_2^*=1$, then
$$E_{\lambda}(\tau )= \{v_2\circ x: x\in F(\tau )\}=\{v_1\circ x: x\in
F(\tau )\}.$$ In either case, ${\mbox{dim}}~(E_{\lambda }(\tau ))=
~\mbox{dim}~(F(\tau )).$
\end{corollary}

\begin{proof}
We may replace $\tau $ by $\theta $ and consider the ordinary
product. Clearly if $x\in F(\theta )$ then $xv_2\in E_{\lambda
}(\theta )$ as $\theta (xv_2)= x. (\lambda v_2)$. Now if $y\in
E_{\lambda }(\theta )$, $y=(yv_1^*).v_2$ and $yv_1^* \in F(\theta
).$ This proves the first part.  Now if there exist $v_1, v_2\in
E_{\lambda }(\theta )$ such that $v_1^*v_2=1.$  If $x_1, \ldots , x_n$
are linearly independent in $F(\theta )$, then $v_2x_1, v_2x_2,
\ldots , v_2x_n$ are linearly independent in $E_{\lambda }(\theta )$
($\sum _ic_iv_2x_i=0$ implies $\sum _ic_ix_i= \sum
_ic_iv_1^*v_2x_i=0$) and conversely if $y_1, y_2, \ldots , y_n$ are
linearly independent in $E_{\lambda }(\theta )$, then $y_1v_1^*,
y_2v_1^*, \ldots , y_nv_1^*$ linearly independent in $F(\theta ).$
This shows ${\mbox{dim}}~(E_{\lambda }(\tau ))= ~\mbox{dim}~(F(\tau
)).$ The proofs of other parts are similar. \end{proof}

\begin{corollary}
If $F(\tau )$ is one dimensional and $E_{\lambda }(\tau )$ is
non-trivial, then $E_{\lambda }(\tau )$ is also one dimensional.
\end{corollary}

\begin{proof}
If $v_1, v_2$ are in $E_{\lambda}(\theta )$, $v_1^*v_2\in F(\theta
)$ and hence it is a scalar.  Now the result is clear from previous
corollary. \end{proof}

The following answers the main question of \cite{FV} for general von
Neumann algebras.

\begin{theorem}\label{automorphism}
Let $\tau :{\mathcal A}\to {\mathcal A}$ be UCP map as above. Then
$x\mapsto \tau (x)$ is an automorphism of the peripheral Poisson
boundary ${\mathcal P} (\tau ).$ The peripheral Poisson boundary of
this automorphism is ${\mathcal P}(\tau ).$
\end{theorem}

\begin{proof}
Consider $x\mapsto \theta (x)$ in the peripheral Poisson boundary of
$\theta .$ Since $\theta $ maps $E(\theta )$ to itself, $\theta $
restricted to the peripheral Poisson boundary is a $*$-endomorphism.
Since $\theta (x)=\lambda .x$ for $x\in E_{\lambda}(\theta )$ it is
surjective.

Now we prove that $\theta $ is isometric on $E(\theta )$, and
consequently it is isometric on ${\mathcal P}(\theta ).$ That will
prove in particular that $\theta $ is injective.  As $\theta $ is a
$*$-endomorphism $\|x\|\geq \|\theta (x) \|\geq \|\theta ^n(x)\|$ for
all $x,n$. Consider $x=\sum _{j=1}^dx_j,$ with $x_j\in E_{\lambda
_j}(\theta )$ for distinct $\lambda _1, \ldots , \lambda _d$ in
${\mathbb T}$ and some $d\geq 1.$ We have  $\theta ^n(x)= \sum
_{j=1}^d\lambda _j^nx_j$ for all $n\geq 1.$ For $\epsilon >0$, by
Lemma \ref{dynamics}, there exists $n$ such that
$$|\lambda _j^n-1|< \frac{\epsilon }{2\sum _{k=1}^d\|x_k\|}, ~~1\leq
j\leq d.$$ Then $\|\theta ^n(x)-x\|= \|\sum _{j=1}^d(\lambda
_j^n-1)x_j\|\leq \frac{\epsilon}{2\sum _{k=1}^d\|x_k\|}.\sum _{j=1}^d\|x_j\|<
\epsilon .$ As this is true for all $\epsilon >0$, we get $\|\theta
(x)\|=\|x\|.$ This proves the result for $\theta .$  Now consider
the compression map $T$ on $E(\theta )$ and we are done. The proof
of the second part is straightforward. \end{proof}

\section{Examples}\label{Examples}

In general it is difficult to explicitly determine Poisson
boundaries of UCP maps. For instance, these von Neumann algebras can
be factors of different types.  There are large number of papers
devoted  to these computations. In this Section we present some examples, deliberately chosen to be simple, to show that the
peripheral Poisson boundary can be very distinct from the Poisson boundary. To begin with  here is a simple finite dimensional example.
\begin{example}
Fix some $d\geq 2$ and let $U$ be a unitary in $M_d({\mathbb C}).$
Let $\tau :M_d({\mathbb C})\to M_d({\mathbb C})$ be the automorphism
$$\tau (X) =U^*XU, ~~X\in M_d({\mathbb C}).$$
The minimal dilation of $\tau $  is $\tau $ itself.  For $\mu \in
{\mathbb C}$, let ${\mathcal H}_{\mu }:=\{h\in {\mathbb C}^d: Uh =
\mu h\}$ be the corresponding eigenspace. Then the fixed point space
of $\tau $ is given by
$$F(\tau )= \oplus _{\mu \in \sigma (U)}B({\mathcal H}_{\mu })$$
and for any $\lambda \in {\mathbb T}$,
$$E_{\lambda }(\tau )=  \{X\in M_d({\mathbb C}): X({\mathcal H}_{\mu })\subseteq {\mathcal H}_{\bar{\lambda } .\mu }, ~~\forall
\mu \}.$$ The Poisson boundary of $\tau $ is the commutant of $U$
and the peripheral Poisson boundary of $\tau $ is whole of
$M_d({\mathbb C}).$ These results can be seen easily by taking $U$
as a diagonal matrix and observing that all matrix units of the
standard basis become eigenvectors of $\tau .$
\end{example}

Here is a class of examples which are not $*$-homomorphisms.

\begin{example} Fix some $d\geq 2$ and let $U,V$ be unitaries in
$M_d({\mathbb C})$ such that $VU= e^{\frac{2\pi i}{d}}UV.$ (Such
unitaries are known as Weyl unitaries.) Take ${\mathcal H}={\mathbb
C}^d\otimes \cdots \otimes {\mathbb C}^d$ ($n$-times, for some $n\in
{\mathbb N}$). Let $p_j, 1\leq j\leq n$ be positive scalars such
that $\sum _jp_j=1. $ Let $U_j=I^{\otimes ^{j-1}}\otimes U\otimes
I^{\otimes ^{n-j}}.$ Define $\tau :B({\mathcal H})\to B({\mathcal
H})$ by
$$\tau (X) = \sum _jp_jU_j^*XU_j.$$
Then we can see that $\tau (V^{\otimes ^n})=e^{\frac{2\pi
i}{d}}V^{\otimes ^n}.$ So the peripheral Poisson boundary  is
distinct from the Poisson boundary.
\end{example}

Perhaps the most well-known example of a noncommutative Poisson
boundary comes from the theory of Toeplitz operators. Understanding
the peripheral Poisson boundary  of this example is very
instructive. Let us recall the setup.

Let ${\mathcal K}=L^2({\mathbb T})$, with the inner product coming
from the normalized Haar measure on the unit circle. It has the
standard orthonormal basis $\{e_n:n\in {\mathbb Z}\}$ where
$$e_n(e^{2\pi it})= e^{2\pi int}, ~~0\leq t<2\pi .$$
Now consider the Hardy space ${\mathcal H}:=H^2({\mathbb T})\subset
L^2({\mathbb T}))$ with its standard orthonormal basis $\{e_n: n\in
{\mathbb Z}_+\}$. Let ${\mathcal A}$ be the von Neumann algebra
$B({\mathcal H}).$  For $f\in L^{\infty}({\mathbb T})$ we have the
multiplication operator on ${\mathcal K}$ defined by
$$M_fh(z)=f(z)h(z), ~~z\in {\mathbb T}.$$
Then the Toeplitz operator with symbol $f\in L^{\infty }({\mathbb
T})$ is the operator
$$T_f= PM_f|_{{\mathcal H}}.$$
where $P$ is the projection of ${\mathcal K}$ onto ${\mathcal H}.$
Let $S$ denote the Toeplitz operator with symbol ``$z$'', which is
of course the unilateral shift on ${\mathcal H}$ mapping $e_n$ to
$e_{n+1}$ for $n\in {\mathbb Z}_+.$ Consider the UCP map $\tau
:B({\mathcal H})\to B({\mathcal H})$ defined by
$$\tau (X)=S^*XS, ~~X\in B({\mathcal H}).$$
By a well known result of Brown and Halmos the space $F(\tau )$ of
fixed points of $\tau $ is precisely the space of Toeplitz operators
$\{T_f: f\in L^{\infty }({\mathbb T}\}.$ Furthermore, the Poisson
boundary structure on this space is given by
$$T_f\circ T_g = T_{fg}, ~~f,g\in L^{\infty}({\mathbb T}).$$
In particular, it is a commutative von Neumann algebra. Now to
describe the peripheral Poisson boundary, let us look at the minimal
dilation.

It is not hard to see that the minimal dilation of $\tau $ is given
by $({\mathcal K}, B({\mathcal K}), \theta )$ where $\theta $ is the
automorphism of $B({\mathcal K})$  given by
$$\theta (Z)= U^*ZU, ~~Z\in B({\mathcal K}),$$
with $U=M_z.$ The space of fixed points of $\theta $ is given by
$\{M_f: f\in L^{\infty }({\mathbb T})\}.$ For $\lambda \in {\mathbb
T}$, consider $V_{\lambda }\in B({\mathcal K})$ defined by
$$V_{\lambda }f(z)= f(\lambda z), ~~f\in {\mathcal K}.$$
Clearly $V_{\lambda }$ is a unitary and $V_{\lambda }e_n = \lambda
^ne_n, ~~\forall n\in {\mathbb Z}.$ It is easy to see that $\theta
(V_{\lambda })=\lambda V_{\lambda },$ that is, $V_{\lambda }\in
E_{\lambda }(\theta ).$ Then by Corollary \ref{isometries},
$$E_{\lambda }(\theta )= \{ M_fV_{\lambda }: f\in
L^{\infty}({\mathbb T})\} = \{ V_{\lambda }M_f: f\in L^{\infty
}({\mathbb T})\}.$$ Then by the main theorem,
$$E_{\lambda }(\tau ) =\{ T_{\lambda , f}: f\in L^{\infty}({\mathbb
T})\},$$ where $T_{\lambda , f }= P_{\mathcal
H}M_fV_{\lambda}|_{\mathcal H}.$ The operators $T_{  \lambda , f}$
can be called $\lambda $-Toeplitz operators, with $\lambda =1$
corresponding to usual Toeplitz operators.  Such operators have been
studied before (See \cite{Ho}). Since $M_fV_{\lambda }.M_gV_{\mu }=
M_{f(\cdot)g(\lambda .\cdot)}V_{\lambda .\mu}$, the product in the
peripheral Poisson boundary is given by
$$T_{\lambda , f}\circ T_{ \mu , g}= T_{\lambda \mu , f(\cdot )g(\lambda .\cdot ) }.$$

Making use of the notion of covariant representations we may express
the observations made above as follows.

\begin{theorem}
Let $\tau:  B(L^2({\mathbb T}))\to B(L^2({\mathbb T}))$ be the map
$X\mapsto S^*XS$ as above. Consider the $C^*$-dynamical system $(
{\mathbb T}, L^{\infty }({\mathbb T}),  \alpha )$, where  ${\mathbb T} \stackrel{\alpha}{\curvearrowright} L^{\infty }({\mathbb T})$ is the natural action
given by $\alpha (\lambda )f(z)= f(\lambda ^{-1} .z)$ for $f\in
L^{\infty }({\mathbb T}), \lambda \in {\mathbb T}.$
 Then the peripheral Poisson boundary of
$\tau $ is equal to the $C^*$-algebra generated by the images of
covariant representation $(\pi , u)$ on  $L^{2}({\mathbb T})$ given
by
$$\pi (f)= M_f,  ~~u (\lambda )= V_{\bar{\lambda }  },~~f\in
L^{\infty}({\mathbb T}), \lambda \in {\mathbb T}.$$
\end{theorem}

\begin{proof}
It is easy to see that these representations are covariant as $$
(V_{\bar{\lambda }}M_fV_{\bar{\lambda } }^*) (h)(z)=
(M_fV_{\bar{\lambda }}^*)h(\bar{\lambda }z) = f(\bar{\lambda }
z)(V_{\bar{\lambda }})^*(h)(\bar{\lambda }z) = f(\bar{\lambda }
z)h(z).$$
Hence $V_{\lambda }M_fV_{\lambda }^*= M_{\alpha({\lambda})(f)}$.
\end{proof}

In the theory of Toeplitz operators the observation that
$$\|T_f\|= \|M_f\|= \|f\|_{\infty },$$
is considered as a very significant and useful result. From Theorem
\ref{main},  we have an extension of this result to have
$$\|PXP\|= \|X\|$$
whenever $X$ is in $\overline{\mbox{span}}\{M_fV_{\lambda }: f\in
L^{\infty }({\mathbb T}), \lambda \in {\mathbb T}\}.$ In particular,
\begin{equation}\label{norm}
	\| \sum _{j=1}^n c_jT_{\lambda _j, f_j}\|= \|\sum _{j=1}^n
	c_jM_{f_j}V_{\lambda _j}\|\end{equation} for $c_j\in {\mathbb C},
f_j\in L^{\infty}({\mathbb T}), \lambda _j\in {\mathbb T}, 1\leq
j\leq n, n\in {\mathbb N}.$

Surprisingly we are unable to take closure in the strong operator
topology as the map $X\mapsto PXP$ won't remain isometric if we do
that. This can be seen as follows. In the  example above, $F(\theta
)=\{M_f:f\in L^\infty(\mathbb{T})\}$ is a von Neumann algebra whose
commutant is itself. Therefore the commutant of  $F(\theta )\bigcup
\{V_{\lambda } \}$ is contained in $L^{\infty }({\mathbb T})$ for
any $\lambda $. Then it is easy to see that if $\lambda = e^{2\pi
i\gamma}$ with $\gamma $ irrational, then this commutant consists of
just scalars. In particular, the von Neumann algebra generated by
$E(\theta )$ is whole of $B({\mathcal K}).$ Obviously the
compression map $X\mapsto PXP$  is not isometric on $B({\mathcal
K}).$ In other words, we are not able to take the peripheral Poisson
boundary as a von Neumann algebra.

The previous example  shows another feature of dilation. If
$|\lambda|<1$ and $X\in E_{\lambda }(\tau )$, it is possible that
there is no $\hat{X}\in E_{\lambda }(\theta )$ such that
$P\hat{X}P=X.$ This follows trivially for this example as an
automorphism being isometric can't have points outside of the unit
circle in its point spectrum.

Now we consider the non-commutative extension of random walks on
discrete groups studied by M. Izumi \cite{Izumi2004}.

\begin{example} Let $G$ be a countable discrete group and let $\mu $
be a probability measure on $G$. The Hilbert space under
consideration is $l^2(G)$ and let $L$, $R $ be respectively the left
and right regular representations of $G$, defined by $L(g)(\delta
_h)= \delta _{gh}$ and $R (g) (\delta _h)= \delta _{hg^{-1}}$ on
standard basis vectors and extended as unitaries. Let $\tau
:B(l^2(G))\to B(l^2(G))$ be defined by
$$\tau (X) = \sum _{g\in G}\mu (g)R(g)XR(g)^*, X\in B(l^2(G)).$$

 Izumi \cite{Izumi2004} has identified the fixed point space of $\tau $
 as a crossed product of classical harmonic elements with $G$.
Let $S$ be the support of $\mu$, $S:=\{g\in G:\mu (g)>0\}.$ Let
$\hat{G}_S$ be the group of all characters on $G$ which are constant
on $S$.
 Now
 consider any $\lambda \in {\mathbb T}$. Suppose there is a character
  $\chi \in \hat{G}_S$ such that \begin{equation}\label{character} \chi (g) =
  \lambda ,~~\forall g\in S.
\end{equation}   Let  $V_{\chi }$ be the unitary defined on $l^2(G)$ by setting $V_{\chi
 }\delta _h= \chi (h)\delta _h$ on basis vectors.
Then $R(g)V_{\chi }R(g)^*\delta _h= R(g)V_{\chi }\delta
_{hg}=R(g)\chi (hg)\delta _{hg}=\chi(hg)\delta _h=\lambda
V_{\chi}\delta _h$ for every $g\in S $. Therefore $V_{\chi }$ is in
$E_{\lambda }(\tau ).$ Then by Lemma \ref{isometries}, we have
$$E_{\lambda }(\tau )= \{ A \circ V_{\chi}: A\in  E_1(\tau )\}= \{V_{\chi
} \circ A: A\in E_1(\tau )\}.$$ We believe that if there is no character
$\chi $ satisfying \ref{character}, then $E_{\lambda}(\tau )=\{0\}$.
\end{example}

\section{Discrete  quantum dynamics}

In quantum theory of open systems dynamics is described through
completely positive maps. In discrete time it just means that we are
studying powers $\tau , \tau ^2, \ldots $ of a single completely
positive maps. Here we study how peripheral eigenspaces move when we
consider powers of a UCP map. Here our focus will be on quantum
channels on the algebra $M_d({\mathbb C})$ of $d\times d$ matrices
and more generally CP maps which preserve a faithful normal state.
However, initial few results are valid for CP maps on general
$C^*$-algebras.

We find the following result from linear algebra very useful. It is
valid even for  infinite dimensional vector spaces. Recall that if
$V_1, \ldots , V_n$ are subspaces of a vector space $V$, the
subspace $W:= \mbox{span}(\bigcup _{j=1}^nV_j)$ is a vector space
direct sum if every vector $w$ in $W$ decomposes uniquely as
$w=v_1+\cdots +v_n$, with $v_j\in V_j.$ We will denote this by
writing $W=\bigvee _{j=1}^nV_j.$

\begin{lemma}\label{generalcase}
Let $V$ be a complex vector space and let $\tau :V\to V$ be a linear
map. Let $p$ be a polynomial of degree $n$ with distinct roots
$\lambda _1, \ldots , \lambda _n$ for some $n\geq 2.$ Let
$E_{\lambda}(\tau )$ be the kernel of $\tau -\lambda I$. Then
$${\mbox{ker}}(p(\tau))= \bigvee _{j=1}^n E_{\lambda _j}(\tau ).$$
\end{lemma}

\begin{proof} Without loss of generality we may assume that $p$ is
monic, so that $p(x)= (x -\lambda _1)\cdots (x-\lambda _n).$ For
computational convenience we wish to assume $\lambda _i\neq 0$ for
every $i$. If not, choose $\mu \in {\mathbb C}$ such that
$\tilde{\lambda }_i:=\lambda _i+\mu \neq 0$ for every $i$. Take
$\tilde{p}(x)= \Pi _i(x-\tilde{\lambda }_i)$ and $\tilde{\tau }=\tau
+\mu I$. Then $\tilde{p}(\tilde {\tau})= p(\tau )$ and
$E_{\tilde{\lambda }_i}(\tilde{\tau })=E_{\lambda _i }(\tau )$ and
so proving the result for $\tilde{p}(\tilde {\tau})$ will suffice.
Therefore we may assume $\lambda _i\neq 0$ for every $i$. Now taking
$d= \Pi_{k=1}^n(-\lambda _k)$, we have the algebraic identity
$$p(x) = d[1- \sum _{i=1}^n\frac{x\Pi _{j\neq i}(x-\lambda _j)}{\lambda _i\Pi _{j\neq i}(\lambda _i-\lambda _j)}],$$
which can be verified easily by evaluating both sides at $0, \lambda
_1, \cdots , \lambda _n.$  Now if $y\in \mbox{ker}(p(\tau ))$,
$p(\tau )y=0$ implies $d[1- \sum _{i=1}^n\frac{\tau \Pi _{j\neq
i}(\tau -\lambda _j)}{\lambda _i\Pi _{j\neq i}(\lambda _i-\lambda
_j)}]y=0.$ Hence $y= \sum _{i=1}^n\frac{\tau \Pi _{j\neq i}(\tau
-\lambda _j)}{\lambda _i\Pi _{j\neq i}(\lambda _i-\lambda _j)}y
=\sum _{i=1}^n y_i,$ where $y_i= \frac{\tau  \Pi _{j\neq i}(\tau
-\lambda _j)}{\lambda _i\Pi (\lambda _i-\lambda _j)}y.$ Note that
$(\tau -\lambda _i).\Pi _{j\neq i}(\tau -\lambda _j)= p(\tau ).$
Hence $(\tau -\lambda _i)y_i=0$. So we get  $y=\sum _iy_i$ with
$y_i\in E_{\lambda _i}(\tau ).$ The converse part namely if $y=\sum
_iy_i$ with $y_i\in E_{\lambda _i}(\tau )$, then $y\in \mbox{ker}
(p(\tau ))$ is obvious, and also in such a case $y_i$ can be
recovered as $\frac{\Pi _{j\neq i}(\tau -\lambda _j)y}{\Pi _{j\neq
i}(\lambda _i-\lambda _j)}$, which proves uniqueness of the
decomposition.
\end{proof}

In the following special case of our current interest on peripheral
spectrum we can do a kind of discrete Fourier inversion.

\begin{lemma}\label{specialcase}
Let $V$ be a complex vector space and let $\tau :V\to V$ be a linear
map. Fix $k\geq 2$ and let $\omega = e^{\frac{2\pi i}{k}}.$ Then
$$F (\tau ^k):=E_1(\tau ^k)= \bigvee _{j=0}^{k-1}E_{\omega ^j}(\tau ).$$ Furthermore,
$x\in F ({\tau ^k})$ decomposes uniquely  as $x= \sum
_{j=0}^{k-1}x_j$ with $x_j\in E_{\omega ^j}(\tau )$ where
\begin{equation}
x_j= \frac{1}{k}\sum _{l=0}^{k-1}\omega ^{-lj}\tau ^l(x).
\end{equation}
\end{lemma}

\begin{proof}
The first part is an immediate consequence of the previous lemma, by
considering the polynomial $p(x)= x^k-1= \Pi _{j=0}^{k-1}(x- \omega
^j).$ For the second part, as $\tau ^k(x)=x$,
$$\tau (x_j)= \sum _{l=0}^{k-1}\omega ^{-lj}\tau ^{l+1}(x)
=\omega ^j\sum _{l=0}^{k-1}\omega ^{-(l+1)j}\tau ^{l+1}(x) =\omega
^j \sum _{l=0}^{k-1}\omega ^{-lj}\tau ^l(x)=\omega ^jx_j.$$ Hence
$x_j\in E_{\omega ^j}(\tau ).$ Further,
$$\sum _{j=0}^{k-1}x_j= \frac{1}{k}\sum _{j,l=0}^{k-1}\omega
^{-lj}\tau ^j(x)= \frac{1}{k}\sum _{j=0}^{k-1}\tau ^{j}(x).(\sum
_{l=0}^{k-1}\omega ^{-lj})=\frac{1}{k}x.k=x,$$ as $\sum
_{l=0}^{k-1}\omega ^{-lj}= \frac{\omega ^{kj}-1}{\omega -1}=0$ for
$j\neq 0.$ \end{proof}

Fix $k\in {\mathbb N}$. Recall the $k$-cyclic Poisson boundary
${\mathcal P}_{G_k}(\tau )$ defined in Definition \label{k-cyclic},
using the subgroup $G_k=\{1, \omega , \ldots , w^{k-1}\}$ of
${\mathbb T}$ where $\omega =e^{\frac{2\pi i}{k}}.$

\begin{theorem}\label{stability}
 Let ${\mathcal A}\subseteq B({\mathcal
H})$ be a von Neumann algebra and let $\tau :{\mathcal A}\to
{\mathcal A}$ be a normal unital completely positive map. Then for
every natural number $k\geq 1$, the peripheral Poisson boundary,
\begin{equation}\label{full} {\mathcal P}(\tau ^k)= {\mathcal P}(\tau ),\end{equation} and the Poisson
boundary \begin{equation}\label{cyclic} {\mathcal P}_{\{1\}}(\tau
^k)= {\mathcal P}_{G_k}(\tau ).\end{equation} In particular,
$k$-cyclic Poisson boundary is a von Neumann algebra for every $k.$
\end{theorem}

\begin{proof}
 Let $({\mathcal K}, {\mathcal B}, \theta )$, be the minimal
dilation of $\tau .$ It suffices to show these equalities for
$\theta $ instead of $\tau $. Then the product $\circ $ is same as
the usual product.

Fix $k\geq 2$. For any $\lambda \in {\mathbb T}$, the polynomial
$p(x)= x^k-\lambda $ has distinct roots, $\{ \mu .\omega ^j: 0\leq
j\leq (k-1)\}$ for some $\mu $ satisfying $\mu ^k=\lambda .$
Therefore, from Lemma \ref{generalcase},
$$E_{\lambda }(\theta ^k )= \bigvee _{j=0}^{k-1}E_{\mu .\omega
^j}(\theta ).$$ As this is true for every $\lambda $ in the circle,
we get $E(\theta ^k)=E(\theta )$. Taking norm closures we get the
first equality. The second part follows in a similar fashion from
Lemma \ref{specialcase}. \end{proof}

It may be noted that equation (\ref{cyclic}) along with Lemma
\ref{specialcase}, tells  us as to how to decompose a vector in the
Poisson boundary of $\tau ^k$  in terms of elements of  $k$-cyclic
peripheral Poisson boundary of $\tau .$

\begin{corollary}

Let ${\mathcal A}$ be a von Neumann algebra and let $\tau :{\mathcal
A}\to {\mathcal A}$ be a normal UCP map. Suppose the group generated
by the peripheral point spectrum of $\tau $ is finite. Then the
peripheral Poisson boundary of $\tau $ is same as the Poisson
boundary of $\tau ^k$ for some $k$. In particular, it is a von
Neumann algebra. \end{corollary}

\begin{proof}
The hypothesis implies that there exists $k\in {\mathbb N}$ such
that every $\lambda \in {\mathbb T}$ with $E_{\lambda }(\tau )$
non-trivial is of the form $e^{\frac{2\pi ij}{k}}$ for some $j.$
Then by the previous theorem ${\mathcal P}(\tau )={\mathcal
P}_{G_k}(\tau )={\mathcal P}_{\{1\}}(\tau ^k).$
\end{proof}

\section*{Appendix}

Here we show that  the compression map $T: \mathcal{P}(\theta) \to
\mathcal{P}(\tau)$ of \Cref{main} is in fact, a complete isometry.
To see this, let us consider the $m^{th}$ ampliation map
$$T^{(m)}:=I_m\otimes T :\mathbb{M}_m \otimes \mathcal{P}(\theta) \to \mathbb{M}_m \otimes \mathcal{P}(\tau).$$
Since $\mathbb{M}_m \otimes E(\theta)$ is dense in $\mathbb{M}_m \otimes \mathcal{P}(\theta)$, it suffices to show that $\norm{ T^{(m)}(Y) } = \norm{Y}$ for every $Y \in \mathbb{M}_m \otimes E(\theta)$.

Let $Y=\sum_{i,j=1}^{m} E_{ij} \otimes y_{ij} \in \mathbb{M}_m
\otimes E(\theta)$, with $y_{ij}\in E(\theta )$, where $\{ E_{ij} :
\leq i,j \leq m \}$ are the matrix units of $\M_m$. For each $1\leq
i,j \leq m$, set $x_{ij}:=p y_{ij} p = T(y_{ij})$ and $X=T^{(m)}
(Y)=\sum_{i,j=1}^{m} E_{ij} \otimes x_{ij}$. As $\norm{I_m \ot p} =
\norm{I_m} \norm{p} = \norm{p} \leq 1$, we have $$\|X\| =
\left\|\sum_{i,j=1}^{m} E_{ij}\otimes (py_{ij}p) \right\| = \left\|
\left(\sum_{i=1}^m E_{ii}\otimes p \right) \left(\sum_{i,j=1}^{m}
E_{ij}\otimes y_{ij} \right) \left(\sum_{i=1}^m E_{ii}\otimes p
\right) \right\| \leq \|Y\|.$$

Suppose $y_{ij}=\sum_{k_{ij}=1}^{d_{ij}}y_{ijk_{ij}}$ where
$y_{ijk_{ij}}\in E_{\lambda_{ijk_{ij}}}(\theta).$ So from
\Cref{TIsOnto} we have  $$y_{ijk_{ij}}=s-\lim_{n\to \infty}
\frac{1}{(\lambda_{ijk_{ij}})^{n}}\theta^{n}(x_{ijk_{ij}}), $$ where
$x_{ijk_{ij}}=py_{ijk_{ij}}p$. Therefore
\begin{align*}
Y= \sum_{i,j=1}^{m} E_{ij} \otimes y_{ij}&= \sum_{i,j=1}^{m} E_{ij} \otimes \left( \sum_{k_{ij}=1}^{d_{ij}}y_{ijk_{ij}} \right)\\
&= s- \lim_{n \to \infty} \sum_{i,j=1}^{m} E_{ij}\otimes \left(\sum_{k_{ij}=1}^{d_{ij}}\frac{1}{(\lambda_{ijk_{ij}})^{n}}\theta^{n}(x_{ijk_{ij}})\right)\\
&=  s- \lim_{n \to \infty} \sum_{i,j=1}^{m} E_{ij}\otimes \theta^{n} \left(\sum_{k_{ij}=1}^{d_{ij}}\frac{1}{(\lambda_{ijk_{ij}})^{n}} x_{ijk_{ij}}\right)\\
&=  s- \lim_{n \to \infty} {(\theta^n)}^{(m)} \left( \sum_{i,j=1}^{m} E_{ij}\otimes \left(\sum_{k_{ij}=1}^{d_{ij}}\frac{1}{(\lambda_{ijk_{ij}})^{n}} x_{ijk_{ij}}\right)\right),
\end{align*}
where ${(\theta^n)}^{(m)} =I_m \otimes \theta^n$ is the $m^{th}$
ampliation of $\theta^n$.

Let $h \in \oplus_{l = 1}^m K$ with $\|h\|=1$ and let $\epsilon > 0$
be arbitrary. We claim that $\norm{Yh} \leq \norm{X} + \epsilon$.
From triangle inequality it follows that
\begin{eqnarray*}
\|Yh\| & \leq & \left\|Yh-  {\theta^n}^{(m)} \left( \sum_{i,j=1}^{m} E_{ij}\otimes \left(\sum_{k_{ij}=1}^{d_{ij}}\frac{1}{(\lambda_{ijk_{ij}})^{n}} x_{ijk_{ij}} \right) \right)h \right\| \\
& + &
\left\|{\theta^n}^{(m)} \left( \sum_{i,j=1}^{m} E_{ij}\otimes \left(\sum_{k_{ij}=1}^{d_{ij}}\frac{1}{(\lambda_{ijk_{ij}})^{n}} x_{ijk_{ij}} \right) \right)h-  {\theta^n}^{(m)} \left( \sum_{i,j=1}^{m} E_{ij}\otimes \left(\sum_{k_{ij}=1}^{d_{ij}} x_{ijk_{ij}} \right) \right)h \right\| \\
& + & \left\|{\theta^n}^{(m)} \left( \sum_{i,j=1}^{m} E_{ij}\otimes \left(\sum_{k_{ij}=1}^{d_{ij}} x_{ijk_{ij}} \right)\right)h \right\|.
\end{eqnarray*}
Using the expression of $Y$ as a certain strong limit, one may find
$N \in \N$ such that  for every $n > N$ the first one of these three
terms is less than $\epsilon/2$. The third term is less than or
equal to $\norm{X}$ as ${\theta^n}^{(m)}$ is a contraction. The
following inequalities explain how one can control the middle term
by repeated application of \Cref{dynamics} to the $m^2$ sequences
determined by each fixed $i$ and $j$.

\begin{eqnarray*}
&& \left\|{\theta^n}^{(m)} \left( \sum_{i,j=1}^{m} E_{ij}\otimes  \left( \sum_{k_{ij}=1}^{d_{ij}}\frac{1}{(\lambda_{ijk_{ij}})^{n}} x_{ijk_{ij}} \right) \right)h -  {\theta^n}^{(m)} \left( \sum_{i,j=1}^{m} E_{ij}\otimes \left( \sum_{k_{ij}=1}^{d_{ij}} x_{ijk_{ij}} \right) \right)h \right\| \\
&=& \left\| {\theta^n}^{(m)} \left( \sum_{i,j=1}^{m} E_{ij} \otimes \left( \sum_{k_{ij}=1}^{d_{ij}}\frac{1}{(\lambda_{ijk_{ij}})^{n}} x_{ijk_{ij}}- x_{ijk_{ij}} \right) \right) h \right\| \\
&\leq&  \left\| \sum_{i,j=1}^{m} E_{ij}\otimes \left( \sum_{k_{ij}=1}^{d_{ij}}\frac{1}{(\lambda_{ijk_{ij}})^{n}} x_{ijk_{ij}}- x_{ijk_{ij}} \right) \right\| \|h \| \\
&\leq& \sum_{i,j=1}^{m} \left\|  E_{ij} \right\| \left\| \left(\sum_{k_{ij}=1}^{d_{ij}}\frac{1}{(\lambda_{ijk_{ij}})^{n}} x_{ijk_{ij}}- x_{ijk_{ij}}\right) \right\| \leq  \sum_{i,j=1}^{m}  \left\| \left(\sum_{k_{ij}=1}^{d_{ij}}\frac{1}{(\lambda_{ijk_{ij}})^{n}} x_{ijk_{ij}}- x_{ijk_{ij}}\right) \right\| \leq \epsilon/2.
\end{eqnarray*}
This shows that $\|Yh\|\leq \|X\| + \epsilon$ and hence $\|Y\|\leq
\|X\| + \epsilon$. As $\epsilon$ was arbitrary, we obtain $\norm{X}
= \norm{Y}$, which proves that $T$ is a complete isometry.

\section*{Acknowledgments}
Bhat gratefully acknowledges funding from  SERB(India) through JC Bose Fellowship No. JBR/2021/000024.
Talwar thanks ISI Bangalore: TSMD/PnC/2021-2022/RA/TSMU-Bangalore/ 009  for financial support through Research Associate scheme.
A part of this work was completed during his research assistantship at Nazarbayev University and he appreciates their support.
Kar is thankful to NBHM (India): 0204/21/2021/R$\&$D-II/9988 for funding.
We thank J. Sarkar for informing us about the literature on $\lambda$-Toeplitz operators.
We extend our gratitude to the reviewer for carefully reading the manuscript and noticing that the map $T$ is a complete isometry.

\end{document}